\documentclass[preprint,12pt]{elsarticle}

\usepackage{mathrsfs}
\usepackage{amssymb}
\usepackage{amsmath}
\usepackage{dsfont}

\usepackage{lineno,hyperref} \modulolinenumbers[5]

\biboptions{compress}
\newtheorem{theorem}{Theorem}[section]
\newtheorem{lemma}[theorem]{Lemma}

\newdefinition{definition}[theorem]{Definition}
\newdefinition{remark}[theorem]{Remark}
\newdefinition{example}[theorem]{Example}
\newdefinition{proposition}[theorem]{Proposition}
\newproof{proof}{Proof}

\journal{}

\begin{document}

\begin{frontmatter}

\title{The fundamental theorem of affine geometry in regular $L^0$-modules}

\author[1]{Mingzhi Wu}
\ead{wumz@cug.edu.cn}

\author[2]{Tiexin Guo\corref{mycorrespondingauthor}}
\cortext[mycorrespondingauthor]{Corresponding author}
\ead{tiexinguo@csu.edu.cn}

\author[2]{Long Long}
\ead{longlong@csu.edu.cn}

\address[1]{School of Mathematics and Physics, China University of Geosciences, Wuhan {\rm 430074}, China}
\address[2]{School of Mathematics and Statistics,
Central South University, Changsha {\rm 410083}, China}

\begin{abstract}
 Let $(\Omega,{\mathcal F},P)$ be a probability space and $L^0({\mathcal F})$ the algebra of equivalence classes of real-valued random variables defined on $(\Omega,{\mathcal F},P)$. A left module $M$ over the algebra $L^0({\mathcal F})$(briefly, an $L^0({\mathcal F})$-module) is said to be regular if $x=y$ for any given two elements $x$ and $y$ in $M$ such that there exists a countable partition $\{A_n,n\in \mathbb N\}$ of $\Omega$ to $\mathcal F$ such that ${\tilde I}_{A_n}\cdot x={\tilde I}_{A_n}\cdot y$ for each $n\in \mathbb N$, where $I_{A_n}$ is the characteristic function of $A_n$ and ${\tilde I}_{A_n}$ its equivalence class. The purpose of this paper is to establish the fundamental theorem of affine geometry in regular $L^0({\mathcal F})$-modules: let $V$ and $V^\prime$ be two regular $L^0({\mathcal F})$-modules such that $V$ contains a free $L^0({\mathcal F})$-submodule of rank $2$, if $T:V\to V^\prime$ is stable and invertible and maps each $L^0$-line segment to an $L^0$-line segment, then $T$ must be $L^0$-affine.
\end{abstract}


\end{frontmatter}

%\linenumbers

\section{Introduction and the main result of this paper}

Throughout this paper, $(\Omega,{\mathcal F},P)$ always denotes a given probability space, $L^0({\mathcal F})$ the algebra of equivalence classes of real-valued random variables defined on $(\Omega,{\mathcal F},P)$, which is endowed with the usual algebraic operations on equivalence classes, and $L^0({\mathcal F}, \mathbb R^n)$ the set of equivalence classes of random vectors from $(\Omega,{\mathcal F},P)$ to the $n$-dimensional Euclidean space $\mathbb R^n$, which forms a free $L^0(\mathcal F)$-module of rank $n$ in a natural way. Randomized versions (or random generalizations) of some fundamental results in classical analysis on $\mathbb R$ or $\mathbb R^n$ are both of interest in their own right and have also been the powerful tools for the study of some important topics in probability theory and stochastic finance in various and unexpected manners. Following are several typical examples.

It is well known from \cite{DS} that $(L^0(\mathcal F), \leq)$ is a Dedekind complete lattice under the partial order: $\xi\leq \eta$ iff $\xi^0(\omega)\leq \eta^0(\omega), a.s.$, where $\xi^0$ and $\eta^0$ are respectively arbitrarily chosen representatives of $\xi$ and $\eta$ in $L^0({\mathcal F})$. The fundamental result can be aptly called the randomized version of order-completeness of the set $\mathbb R$ of real numbers, which together with whose equivalent variant (often called the essential supremum and infimum theorem on the set of real-valued random variables) has been frequently used in probability theory and stochastic control \cite{Yan,YPFW}. Another typical example from \cite{KS} is the randomized version of the classical Bolzano-Weierstrass theorem, which states that there exists a random subsequence $\{x_{n_k},k\in \mathbb N\}$ converging a.s. for an arbitrary almost surely (a.s.) bounded sequence $\{x_n,n\in \mathbb N\}$ in $L^0({\mathcal F}, \mathbb R^n)$ (namely, $\sup_n|x_n(\omega)|<+\infty, a.s.$), where $\{n_k,k\in\mathbb N\}$ is a sequence of positive integer-valued random variables defined on $(\Omega,{\mathcal F},P)$ such that $n_k(\omega)<n_{k+1}(\omega)$ for each $\omega\in \Omega$ and each $k\in \mathbb N$, and $x_{n_k}=\sum^\infty_{l=1}\tilde I_{(n_k=l)}\cdot x_l$ with $\tilde I_{(n_k=l)}$ being the equivalence class of the characteristic function $I_{(n_k=l)}$ of $(n_k=l):=\{\omega\in \Omega: n_k(\omega)=l\}$ for any $k$ and $l$ in $\mathbb N$. The use of the fundamental result has considerably simplified the proof of no-arbitrage criteria \cite{KS}, and in particular the randomized version has motivated the fruitful development of the theory of random sequential compactness in random normed modules (briefly, $RN$-modules) \cite{GWXY}.

$RN$-modules are a random generalization of ordinary normed spaces. In fact, random functional analysis is just based on such an idea of randomizing the classical space theory in functional analysis and has been systematically developed as a whole, while random convex analysis, as an important part of random functional analysis, and its applications to conditional convex risk measures, random control and random optimization have been also developed, see \cite{Guo-JFA,GS,Wu,GZZ-SCM,GZWYYZ,GZWG,GZWW} and the references therein for details.

In 2009, Artstein-Avidan and Milman \cite {AM} characterized the fully order preserving (reversing) operators on the set of lower semicontinuous convex functions on $\mathbb R^n$. Further, in 2015 Insem, Reem and Svaiter \cite{IRS} extended the main results in \cite{AM} from $\mathbb R^n$ to real Banach spaces in a nontrivial manner. The results both in \cite {AM} and in \cite{IRS} are based on the fundamental theorem of affine geometry and classical convex analysis. Our future purpose is to generalize the main results in \cite{IRS} from Banach spaces to complete $RN$-modules. Now that random convex analysis has been established as stated above, it remains to generalize the fundamental theorem of affine geometry from linear spaces to left modules over the algebra $L^0({\mathcal F})$ (briefly, $L^0({\mathcal F})$-modules), which is just the goal of this paper.

Propositions \ref{fdag} and \ref{ftag} below survey the finite-dimensional and the infinite-dimensional versions of the fundamental theorem of affine geometry, respectively.

\begin{proposition}\label{fdag}( see \cite{Artin,CP} )
  Let $V$ and $V^\prime$ be two finite dimensional real vector spaces with $dim(V)\geq 2$. If $T: V\to V^\prime$ is invertible and maps each line segment to a line segment, then $T$ is affine.
\end{proposition}

\begin{proposition}\label{ftag}( see \cite{IRS} )
 Let $V$ and $V^\prime$ be any two real vector spaces with $dim(V)\geq 2$. If $T: V\to V^\prime$ is invertible and maps each line segment to a line segment, then $T$ is affine.
\end{proposition}

To state our central result, let us first introduce some terminologies as follows.

\begin{definition}
 Let $V$ and $V^\prime$ be any two $L^0({\mathcal F})$-modules. A mapping $T:V\to V^\prime$ is said to be:\\
 (1) $L^0$-linear if $T$ is a module homomorphism from $V$ to $V^\prime$.\\
 (2) $L^0$-affine if there exist an $L^0$-linear mapping $S$ from $V$ to $V^\prime$ and $b\in V^\prime$ such that $T(x)=S(x)+b$ for any $x\in V$.\\
 (3) Stable if $T(\tilde I_Ax+\tilde I_{A^c}y)=\tilde I_A T(x)+\tilde I_{A^c}T(y)$ for any $x$ and $y$ in $V$ and any $A$ in $\mathcal F$, where $\tilde I_A$ the equivalence class of the characteristic function $I_A$ of $A$ (in Section 2 of this paper, we will check that $T$ is stable iff $T$ has the local property, namely $\tilde I_AT(x)=\tilde I_AT(\tilde I_A x)$ for any $x\in V$ and $A\in \mathcal F$).\\
Besides, $T$ is said to map each $L^0$-line segment to an $L^0$-line segment if $T([x,y])=[T(x),T(y)]$ for any two different elements $x$ and $y$ in $V$, where $[x,y]=\{\lambda x+(1-\lambda)y:\lambda\in L^0({\mathcal F}) \mbox {~and~}0\leq \lambda\leq 1\}$, called the $L^0$-line segment linking $x$ and $y$. Further, we say that $V$ contains a free $L^0({\mathcal F})$-submodule of rank $2$ if $V$ contains two $L^0({\mathcal F})$-independent elements $x$ and $y$, which means $\xi=\eta=0$ for any two $\xi$ and $\eta$ in $L^0({\mathcal F})$ such that $\xi x+\eta y=0$.
\end{definition}

When we generalize Proposition \ref{ftag} stated above to $L^0({\mathcal F})$-modules, we require that the $L^0({\mathcal F})$-modules have the following regular property so that we can introduce the notion of a support for an element in them, which will play a crucial role in establishing our main result.

\begin{definition}\label{regular}(see \cite{Guo-JFA})
An $L^0({\mathcal F})$-module $V$ is said to be regular if it always holds that $x$ and $y$ are equal for any given $x$ and $y$ in $V$ such that there exists a countable partition $\{A_n,n\in \mathbb N\}$ of $\Omega$ to $\mathcal F$ (namely each $A_n\in \mathcal F$, and $A_n\cap A_m=\emptyset$ for any $n\neq m$ and $\cup_n A_n=\Omega$) such that ${\tilde I}_{A_n}\cdot x={\tilde I}_{A_n}\cdot y$ for each $n\in \mathbb N$.
\end{definition}

As mentioned in \cite{Guo-JFA}, the regular requirement as in Definition \ref{regular} is merely a slight restrictive condition since $RN$-modules and random locally convex modules as the central framework of random functional analysis are all regular!

We can now state the central result of this paper as follows.

\begin{theorem}\label{main}
Let $V$ and $V^\prime$ be any two regular $L^0({\mathcal F})$-modules such that $V$ contains a free $L^0({\mathcal F})$-submodule of rank $2$. If $T:V\to V^\prime$ is stable, invertible and maps each $L^0$-line segment to an $L^0$-line segment, then $T$ is $L^0$-affine.
\end{theorem}

When $(\Omega,{\mathcal F},P)$ is trivial, namely ${\mathcal F}=\{\Omega,\emptyset\}$, Theorem \ref{main} immediately reduces to Proposition \ref{ftag} since in the case an $L^0({\mathcal F})$-module is automatically a real vector space and thus regular, and a mapping between two real vector spaces is also automatically stable.

Compared with Proposition \ref{ftag}, in the general case the unique additional condition is stability of $T$ and an example will be further provided to show that stability of $T$ is essential in Theorem \ref{main}. On the other hand, the algebra $L^0({\mathcal F})$, unlike the real number field with dimension one, may be infinite-dimensional so that the $L^0({\mathcal F})$-modules have more complicated algebraic structure than ordinary real linear spaces, for example, an $L^0$-line $l(x,y):=\{\lambda x+(1-\lambda)y: \lambda\in L^0(\mathcal F)\}$ can not be uniquely determined by any two points $u$ and $v$ on it, which makes the proof of Theorem \ref{main} much more involved than the proofs of Propositions \ref{fdag} and \ref{ftag}.

Since the fundamental theorem of affine geometry is of fundamental importance, it has been generalized from linear spaces to free modules over some kinds of rings \cite{LK}, where some other types of conditions are assumed. However, our Theorem \ref{main} is not a special case of the fundamental theorem of affine geometry given in \cite{LK}, at least we do not require that $V$ and $V'$ in our Theorem \ref{main} be free, the related discussions will also be given in Section 3 of this paper. Besides, it should also be pointed out that owing to the peculiarity of $L^0({\mathcal F})$ our Theorem \ref{main} is not merely more concise and in particular meets the needs of our forthcoming work.

\par
The remainder of this paper is organized as follows. Section 2 is devoted to some preliminaries and Section 3 to the proof of Theorem \ref{main} while some other main results of independent interest are also given.

\section{Preliminaries}

For each $A\in {\mathcal F}$, the equivalence class of $A$ refers to $\tilde A=\{B\in {\mathcal F}: P(A\triangle B)=0\}$, and we often also write $I_{\tilde A}$ for ${\tilde I}_A$. For any $A$ and $B$ in $\mathcal F$, $\tilde A\subset \tilde B$ means that $P(A\setminus B)=0$. Given $\xi\in L^0({\mathcal F})$, let $\xi^0(\cdot)$ be an arbitrarily chosen representative of $\xi$. We write $[\xi\neq 0]$ for the equivalence class of the measurable set $\{\omega\in\Omega: \xi^0(\omega)\neq 0\}$. The statement ``$\xi\neq 0$ on $\Omega$'' means that $\xi^0(\omega)\neq 0$, $P$-a.s., in other words, $\xi$ is an invertible element of the algebra $L^0({\mathcal F})$. The generalized inverse $\xi^{-1}$ of $\xi$ is defined as the equivalent class of the random variable $(\xi^0)^{-1}$ defined by
$$(\xi^0)^{-1}(\omega)=\left\{
                         \begin{array}{ll}
                           (\xi^0(\omega))^{-1}, &\hbox {if $\xi^0(\omega)\neq 0$},  \\
                           0, & \hbox{if $\xi^0(\omega)= 0$}.
                         \end{array}
                       \right.
$$
It is clear that $\xi \cdot \xi^{-1}=I_{[\xi\neq 0]}$. Some other notations such as $[\xi=0]$ and ``$\xi>0$ on $\Omega$'' are understood in a similar way.

Now we introduce the notion of a support for an element in a regular $L^0({\mathcal F})$-module. Since the notion depends on what has been called the randomized version of order-completeness of the set $\mathbb R$ of real numbers in Section 1, we restate the result as follows for the reader's convenience.

\begin{proposition}\label{esssup}(see \cite{DS})
$(L^0(\mathcal F), \leq)$ is a Dedekind complete lattice, that is to say, every nonempty set $H$ in $L^0(\mathcal F)$ with an upper bound (a lower bound) has a supremum (accordingly, an infimum), denoted by $\vee H$ (accordingly, $\wedge H$), and if in addition $H$ is directed upward (accordingly, downward), then there exists a nondecreasing (accordingly, nonincreasing) sequence $\{\xi_n,n\in \mathbb N\}$ in $H$ such that $\vee\{\xi_n,n\in \mathbb N\}=\vee H$ (accordingly, $\wedge \{\xi_n,n\in \mathbb N\}=\wedge H$).
\end{proposition}

Let $V$ be an $L^0({\mathcal F})$-module and $x$ an element of $V$. Denote ${\mathcal A}=\{B\in {\mathcal F}: \tilde I_B x=0\}$. Since $\emptyset\in {\mathcal A}$ and $\tilde I_B\leq 1$ for any $B\in \mathcal F$, $\{\tilde I_B: B\in {\mathcal A}\}$ has a supremum. For any $B_1$ and $B_2$ in ${\mathcal A}$, we have $\tilde I_{B_1\cup B_2}x=(\tilde I_{B_1}+\tilde I_{B_2\setminus B_1})x=\tilde I_{B_1} x+\tilde I_{B_2\setminus B_1} x=0$, thus $B_1\cup B_2$ belongs to ${\mathcal A}$, namely ${\mathcal A}$ is directed upward. Therefore according to Proposition \ref{esssup}, there exists a non-decreasing sequence $\{B_n,n\in \mathbb N\}$ in ${\mathcal A}$ such that $\vee \{\tilde I_{B_n},n\in \mathbb N\}=\vee \{\tilde I_B: B\in \mathcal A\}$. Let $A=\cup^{\infty}_{n=1} B_n$, then we see that $\vee \{\tilde I_B: B\in \mathcal A\}=\vee \{\tilde I_{B_n},n\in \mathbb N\}=\tilde I_A$. Set $A_0=A^c$, $A_1=B_1$ and $A_n=B_n\setminus B_{n-1}$ for $n\geq 2$, then $\{A_n: n=0,1,2,\dots\}$ is a countable partition of $\Omega$ to $\mathcal F$ such that $\tilde I_{A_n}\cdot \tilde I_A x=0$ for every $A_n$. However we cannot deduce that $I_A x=0$ if no other conditions are assumed. Thus there is a minor mistake in the proof of Theorem 1.1 of \cite{GS}. If we add the requirement that the $L^0({\mathcal F})$-module $V$ be regular, then we can conclude that $I_A x=0$ since $\tilde I_{A_n}\cdot 0=0$ and $\tilde I_{A_n}\cdot \tilde I_A x=0$ for every $A_n$. In such a case, we write $[x=0]$ for the equivalence class $\tilde A$, and $[x\neq 0]$ for the equivalence class $\tilde {A^c}$. Besides, $A^c$ is said to be a support for $x$, and when $P(A^c)=1$ we call $x$ having full support.

\begin{proposition}
 Let $V$ be a regular $L^0({\mathcal F})$-module and $x$ an element of $V$ with full support. Then the following statements are true:\\
(1) If $\xi x=0$ for some $\xi\in L^0({\mathcal F})$, then $\xi=0$. \\
(2) If $\xi\in L^0({\mathcal F})$ is invertible, then $\xi x$ has full support.\\
(3) If $\xi$ and $\eta$ are elements in $L^0({\mathcal F})$ such that $\xi x=\eta x$, then $\xi=\eta$.\\
(4) If $y$ and $z$ are two $L^0({\mathcal F})$-independent elements in $V$, then both $y$ and $z$ have full support.
\end{proposition}

\begin{proof}
(1) Since $0=\xi^{-1}\cdot \xi x=I_{[\xi\neq 0]}x$ and $x$ has full support, we obtain $I_{[\xi\neq 0]}=0$, namely $\xi=0$.

(2) If $\tilde I_A\cdot \xi x=\tilde I_A\xi \cdot x=0$ for some $A\in \mathcal F$, then it follows from (1) that $\tilde I_A\xi=0$, thus $\tilde I_A=0$, meaning $\xi x$ has full support.

(3) It immediately follows from (1).

(4) Since $I_{[y=0]}y+I_{[z=0]}z=0+0=0$, and $y$ and $z$ are $L^0({\mathcal F})$-independent, we obtain that $I_{[y=0]}=I_{[z=0]}=0$, namely both $y$ and $z$ have full support.

This completes the proof.
\end{proof}

By adding the requirement that the $L^0({\mathcal F})$-module be regular, we give a modification of Theorem 1.1 of \cite{GS} as follows, which will be used in the proof of Proposition \ref{line-line}.

\begin{proposition}\label{atfgm}
Let $V$ be a finitely generated regular $L^0({\mathcal F})$-module. Then there exists a finite partition $\{A_0, A_1, \dots, A_n\}$ of $\Omega$ to $\mathcal F$ such that $\tilde I_{A_i}V$ is a free module of rank $i$ over the algebra $\tilde I_{A_i}L^0({\mathcal F})$ for each $i\in \{0, 1, 2, \dots,n\}$ satisfying $P(A_i) >0$, in which case $V=\bigoplus^n_{i=0}\tilde I_{A_i}V$ and each such $A_i$ is unique in the sense of almost sure equality.
\end{proposition}

Next, we study the relations among $L^0$-affine property, stability and local property of a mapping.

\begin{proposition}\label{aff-lc}
 Let $V$ and $V^\prime$ be two $L^0({\mathcal F})$-modules and $T$ a mapping from $V$ to $V^\prime$. Then the following statements are true: \\
 (1). $T$ is $L^0$-affine iff $T[\lambda x+(1-\lambda) y]=\lambda T(x)+(1-\lambda)T(y)$ for any $x$ and $y$ in $V$ and any $\lambda$ in $L^0({\mathcal F})$.\\
 (2). If $T$ is $L^0$-affine, then $T$ must be stable.\\
 (3). If $T$ is stable and $T(0)=0$, then $T({\tilde I}_A x)={\tilde I}_A T(x)$ for any $x\in V$ and $A\in {\mathcal F}$.\\
 (4). $T$ is stable iff $T$ has the local property.\\
 (5). If $T$ is stable and bijective, then $T^{-1}$ is also stable.\\
\end{proposition}

\begin{proof}
(1). Define $S: V\to V^\prime$ by $S(x)=T(x)-T(0), \forall x\in V$. Then $T$ is $L^0$-affine iff $S$ is $L^0$-linear.

If $T$ is $L^0$-affine, then $S$ is $L^0$-linear, thus for any $x$ and $y$ in $V$ and any $\lambda$ in $L^0({\mathcal F})$, we have
\begin{eqnarray*}
T[\lambda x+(1-\lambda) y]&=& S[\lambda x+(1-\lambda) y]+T(0)\\
                          &=& \lambda S(x)+(1-\lambda)S(y)+T(0)\\
                          &=& \lambda [S(x)+T(0)]+(1-\lambda)[S(y)+T(0)]\\
                          &=& \lambda T(x)+(1-\lambda)T(y).
\end{eqnarray*}

Conversely, if $T[\lambda x+(1-\lambda) y]=\lambda T(x)+(1-\lambda)T(y)$ for any $x$ and $y$ in $V$ and any $\lambda$ in $L^0({\mathcal F})$, then for any $x\in V$ and $\xi\in L^0({\mathcal F})$, we have
$$T(\xi x)=T[\xi x+(1-\xi)\cdot 0]=\xi T(x)+(1-\xi)T(0),$$
thus $$ S(\xi x)=T(\xi x)-T(0)=\xi[T(x)-T(0)]=\xi S(x),$$
and for any $x$ and $y$ in $V$, we have \begin{eqnarray*}
          T(x+y)&=& \frac{1}{2}T(2x)+\frac{1}{2}T(2y)  \\
                &=& \frac{1}{2}[T(2x-0)+T(2y-0)]\\
                &=& \frac{1}{2}[2T(x)-T(0)+2T(y)-T(0)]\\
                &=& T(x)+T(y)-T(0),
        \end{eqnarray*}
thus $$S(x+y)=T(x+y)-T(0)=T(x)-T(0)+T(y)-T(0)=S(x)+S(y).$$
Therefore $S$ is $L^0$-linear, equivalently $T$ is $L^0$-affine.

(2). If $T$ is $L^0$-affine, then by (1), $T(\tilde I_A x+\tilde I_{A^c}y)=\tilde I_A T(x)+\tilde I_{A^c}T(y)$ for any $x$ and $y$ in $V$ and any $A$ in $\mathcal F$, namely $T$ is stable.

(3). If $T$ is stable and $T(0)=0$, then for any $x\in V$ and $A\in {\mathcal F}$, we have $$T({\tilde I}_A x)=T({\tilde I}_A x+{\tilde I}_{A^c}\cdot 0)={\tilde I}_AT(x)+ {\tilde I}_{A^c}T(0)={\tilde I}_AT(x).$$

(4). If $T$ is stable, then for any $x\in V$ and $A\in {\mathcal F}$, $T({\tilde I}_A x)=T({\tilde I}_A x+{\tilde I}_{A^c}\cdot 0)={\tilde I}_AT(x)+ {\tilde I}_{A^c}T(0)$, thus ${\tilde I}_AT({\tilde I}_A x)={\tilde I}_AT(x)$, namely $T$ has the local property.

 Conversely, if $T$ has the local property, then for any $x$ and $y$ in $V$ and any $A$ in ${\mathcal F}$,
\begin{eqnarray*}
T({\tilde I}_A x+{\tilde I}_{A^c}y)&=& {\tilde I}_AT({\tilde I}_A x+{\tilde I}_{A^c}y)+ {\tilde I}_{A^c}T({\tilde I}_A x+{\tilde I}_{A^c}y) \\
                                   &=& {\tilde I}_AT[{\tilde I}_A({\tilde I}_A x+{\tilde I}_{A^c}y)]+{\tilde I}_{A^c}T[{\tilde I}_{A^c}({\tilde I}_A x+{\tilde I}_{A^c}y)]\\
                       & =& {\tilde I}_AT({\tilde I}_A x)+ {\tilde I}_{A^c}T({\tilde I}_{A^c}y) \\
                       & =& {\tilde I}_AT(x)+{\tilde I}_{A^c}T(y),
\end{eqnarray*}
thus $T$ is stable.

(5). If $T$ is stable and bijective, then for any $y\in V^\prime$ and $A\in {\mathcal F}$, we have
\begin{align*}
  T[{\tilde I}_A T^{-1}(y)] & =T[{\tilde I}_A T^{-1}(y)+{\tilde I}_{A^c}\cdot 0] \\
   & ={\tilde I}_A T[T^{-1}(y)]+{\tilde I}_{A^c}T(0) \\
   & ={\tilde I}_Ay+{\tilde I}_{A^c}T(0),
\end{align*}
and
\begin{align*}
  T[{\tilde I}_AT^{-1}({\tilde I}_Ay)] & =T[{\tilde I}_A T^{-1}({\tilde I}_Ay)+{\tilde I}_{A^c}\cdot 0] \\
   & ={\tilde I}_AT[T^{-1}({\tilde I}_Ay)]+{\tilde I}_{A^c}T(0) \\
   & ={\tilde I}_Ay+{\tilde I}_{A^c}T(0),
\end{align*}
thus $T[{\tilde I}_A T^{-1}(y)]=T[{\tilde I}_AT^{-1}({\tilde I}_Ay)]$. Since $T$ is injective, we obtain that ${\tilde I}_A T^{-1}(y)={\tilde I}_AT^{-1}({\tilde I}_Ay)$, namely $T^{-1}$ has the local property. It follows from (4) that $T^{-1}$ is stable.

This completes the proof.
\end{proof}

Lemma \ref{id} below states that the identity mapping is the unique endomorphism on $L^0({\mathcal F})$ with the local property. It will be used in the proof of Proposition \ref{line-line2}.

As usual, we always regard $\mathbb R$ as a subset of $L^0(\mathcal{F})$ by identifying every element $r\in \mathbb R$ with the equivalence class of the constant function with value $r$.

\begin{lemma}\label{id}
Let $\phi: L^0({\mathcal F}) \to L^0({\mathcal F})$ be a mapping such that:\\
(1). $\phi$ is local;\\
(2). $\phi(\xi+\eta)=\phi(\xi)+\phi(\eta), \forall \xi,\eta\in L^0({\mathcal F})$;\\
(3). $\phi(\xi\eta)=\phi(\xi)\phi(\eta),\forall \xi,\eta\in L^0({\mathcal F})$;\\
(4). $\phi(1)=1$.\\
Then $\phi$ is the identity mapping, namely $\phi(\xi)=\xi, \forall \xi\in L^0({\mathcal F})$.
\end{lemma}

\begin{proof}
From (2), $\phi(0)=2\phi(0)$, thus $\phi(0)=0$. Since $\phi(\xi-\xi)=\phi(0)=\phi(\xi)+\phi(-\xi)$, we obtain that $\phi(-\xi)=-\phi(\xi)$ for every $\xi\in L^0({\mathcal F})$. Since $\phi(1)=1$, it is easy to deduce that $\phi(p)=p$ for any integer $p$ and further $\phi(r)=r$ for any rational number $r$.

Let $\xi$ and $\eta$ be two elements in $L^0({\mathcal F})$ with $\xi \geq \eta$, then from (3) we obtain that $\phi(\xi-\eta)=\phi(\sqrt {\xi-\eta}\sqrt {\xi-\eta})=\phi(\sqrt {\xi-\eta})\phi(\sqrt {\xi-\eta})\geq 0$. Since $\phi(\xi-\eta)=\phi(\xi)+\phi(-\eta)=\phi(\xi)-\phi(\eta)$, it follows that $\phi(\xi)\geq \phi(\eta)$, therefore $\phi$ is monotonically increasing. Since $\phi(r)=r$ for any rational number $r$, from the monotonicity of $\phi$ we obtain that $\phi(r)=r$ for any real number $r$.

For any simple element $q=\sum^d_{i=1}r_i{\tilde I}_{A_i}$ in $L^0({\mathcal F})$ (where each $r_i$ is a real number and $\{A_i:i=1,\dots, d\}$ is a finite partition of $\Omega$ to $\mathcal F$), by $(4)$ of Proposition \ref{aff-lc} we have $\phi(q)=\sum^d_{i=1}{\tilde I}_{A_i}r_i=q$.

Given any $\xi\in L^0({\mathcal F})$ satisfying that there exists a positive $r\in \mathbb R$ such that $|\xi|\leq r$. Let $q_-=\sum^d_{i=1}r_i{\tilde I}_{A_i}$ and $q_+=\sum^k_{j=1}t_j{\tilde I}_{B_j}$ be any two simple elements in $L^0({\mathcal F})$ such that $q_-\leq \xi\leq q_+$, then using the monotonicity of $\phi$, we have $q_-=\phi(q_-)\leq \phi(\xi)\leq \phi(q_+)=q_+$. Taking all such possible $q_-$ and $q_+$, we thus obtain that $\phi(\xi)=\xi$.

For a general $\xi\in L^0({\mathcal F})$, let $A_n=[n-1\leq |\xi| <n]$ for each $n\in\mathbb N$, then $\phi(I_{A_n}\xi)=I_{A_n}\xi$ since $|I_{A_n}\xi|\leq n$. From (4) and (3) of Proposition \ref{aff-lc} we get $\phi(I_{A_n}\xi)=I_{A_n}\phi(\xi)$, again by noting that $\sum^\infty_{n=1}I_{A_n}=1$ we can thus obtain that $\phi(\xi)=\sum^\infty_{n=1}I_{A_n}\phi(\xi)=\sum^\infty_{n=1}\phi(I_{A_n}\xi)=\sum^\infty_{n=1}I_{A_n} \xi=\xi$.

This completes the proof.
\end{proof}

\section{Main results and their proofs}

To prove Theorem \ref{main}, we first need to establish the $L^0$-affineness of a mapping that maps $L^0$-lines to $L^0$-lines.

\begin{definition}
Let $V$ and $V^\prime$ be two $L^0({\mathcal F})$-modules, a mapping $T: V\to V^\prime$ is said to map each $L^0$-line to an $L^0$-line if for any two points $x$ and $y$ in $V$ there exist two points $u$ and $v$ in $V^\prime$ such that $T(l(x,y))=l(u,v)$.
\end{definition}

Let $V$ and $V^\prime$ be two $L^0({\mathcal F})$-modules and $T: V\to V^\prime$ an injective mapping. According to (1) of Proposition \ref{aff-lc}, if $T$ is $L^0$-affine then $T$ maps each $L^0$-line to an $L^0$-line. Proposition \ref{line-line} below states that the converse is also true when $V$ and $V^\prime$ are regular and $V$ contains a free $L^0({\mathcal F})$-submodule of rank 2. This result may also be called the fundamental theorem of affine geometry in regular $L^0$-modules. To prove Proposition \ref{line-line}, we will first show a relatively simple version--Proposition \ref{line-line2} below, in which $V$ itself is assumed to be a free module of rank 2.

 Before we state and prove Propositions \ref{line-line2} and \ref{line-line}, we make a comparison between straight lines in a real vector space and $L^0$-lines in an $L^0({\mathcal F})$-module. In a real vector space, any two different points in a given straight line determine the same straight line, whereas in an $L^0({\mathcal F})$-module $V$, if $w$ and $z$ are two different points in the $L^0$-line $l(x,y)$, the $L^0$-line $l(w,z)$ may be not the same as $l(x,y)$. For instance, let $x\in V$ be a nonzero element, then for any $A\in {\mathcal F}$, ${\tilde I}_Ax$ lies in the $L^0$-line $l(0,x)=\{\lambda x: \lambda\in L^0({\mathcal F})\}$, however if ${\tilde I}_Ax$ is nonzero, then the $L^0$-line $l(0,{\tilde I}_Ax)=\{{\tilde I}_A\lambda x: \lambda\in L^0({\mathcal F})\}$ is probably not the same as $l(0,x)$. Thus we should be careful when we handle problems involving $L^0$-lines.

 Since the mappings involved in the main results of this paper are injective, Proposition \ref{line-line-exact} below shows that when one wants to prove that an injective mapping $T$ maps an $L^0$-line to an $L^0$-line, he need only prove $T(l(x,y))=l(T(x),T(y))$ for any two different $x$ and $y$, since this is obvious for $x=y$.

\begin{proposition}\label{line-line-exact}
  Let $V$ and $V^\prime$ be two regular $L^0({\mathcal F})$-modules. If $T: V\to V^\prime$ is stable, injective and maps each $L^0$-line to an $L^0$-line, then $T(l(x,y))=l(T(x),T(y))$ for any two different points $x$ and $y$ in $V$.
\end{proposition}

\begin{proof}
Define $S: V\to V^\prime$ by $S(x)=T(x)-T(0), \forall x\in V$, then $S(0)=0$. With the assumptions on $T$, we obtain that $S$ is stable, injective and maps each $L^0$-line to an $L^0$-line. Fix any two points $x$ and $y$ in $V$, it remains to show that $S(l(x,y))=l(S(x),S(y))$. Suppose that $u$ and $v$ are in $V^\prime$ such that $S(l(x,y))=l(u,v)$. Since $l(S(x),S(y))\subset l(u,v)$ is obvious, it suffices to show that $l(u,v)\subset l(S(x),S(y))$.

Using (3) of Proposition \ref{aff-lc}, we obtain that $$\{B\in {\mathcal F}: \tilde I_B[S(x)-S(y)]=0\}=\{B\in {\mathcal F}: S(\tilde I_Bx)-S(\tilde I_By)=0\},$$
 and by the injectivity of $S$ we obtain that $$\{B\in {\mathcal F}: S(\tilde I_Bx)-S(\tilde I_By)=0\}=\{B\in {\mathcal F}: \tilde I_Bx-\tilde I_By=0\}.$$ Thus $\{B\in {\mathcal F}: \tilde I_B[S(x)-S(y)]=0\}=\{B\in {\mathcal F}: \tilde I_B(x-y)=0\}$, which implies that $[x-y=0]=[S(x)-S(y)=0]$, equivalently $[x-y\neq 0]=[S(x)-S(y)\neq 0]$.

Denote $A=[x-y\neq 0]$ and $A^c=[x-y=0]$. Since we have shown that $[S(x)-S(y)\neq 0]=A$, we will show that $[u-v\neq 0]=A$.

Let $\xi_1$ and $\eta_1$ be elements in $L^0({\mathcal F})$ such that $u=S[\xi_1 x+(1-\xi_1) y]$ and $v=S[\eta_1 x+(1-\eta_1)y]$. Since $ I_{A^c}x=I_{A^c}y$, we obtain that $I_{A^c}[\xi_1 x+(1-\xi_1) y]=\xi_1 I_{A^c} x+(1-\xi_1)I_{A^c} y=I_{A^c} x$, and similarly $I_{A^c}[\eta_1 x+(1-\eta_1) y]=I_{A^c} x$. Then using (3) of Proposition \ref{aff-lc}, we obtain that $I_{A^c}u=I_{A^c}S[\xi_1 x+(1-\xi_1) y]=S(I_{A^c} x)$, and similarly $I_{A^c}v=S(I_{A^c} x)$. Therefore $I_{A^c}u=I_{A^c}v$, implying that $A^c\subset [u-v=0]$.

Let $\xi$ and $\eta$ be elements in $L^0({\mathcal F})$ such that $S(x)=\xi u+(1-\xi) v$ and $S(y)=\eta u+(1-\eta) v$, then $S(x)-S(y)=(\xi-\eta)(u-v)$. For any $B\in {\mathcal F}$ satisfying $\tilde I_B(u-v)=0$, we have $\tilde I_B[S(x)-S(y)]=(\xi-\eta)\tilde I_B(u-v)=0$, thus $[u-v=0]\subset [S(x)-S(y)=0]=A^c$. Similarly we deduce that $[\xi-\eta=0]\subset [S(x)-S(y)=0]=A^c$, equivalently $A\subset [\xi-\eta\neq 0]$. Combining $A^c\subset [u-v=0]$ and $[u-v=0]\subset A^c$, we obtain that $[u-v=0]=A^c$, equivalently $[u-v\neq 0]=A$.

Now we can deduce that $u-v=I_A(u-v)=I_AI_{[\xi-\eta\neq 0]}(u-v)=I_A(\xi-\eta)^{-1}(\xi-\eta)(u-v)=(\xi-\eta)^{-1}I_A[S(x)-S(y)]=(\xi-\eta)^{-1}[S(x)-S(y)]$.

 Combining $S(x)=\xi u+(1-\xi) v$ and $u-v=(\xi-\eta)^{-1}[S(x)-S(y)]$, we obtain that $v=[1-\xi(\xi-\eta)^{-1}]S(x)+\xi(\xi-\eta)^{-1}S(y)$ and $u=[1-\xi(\xi-\eta)^{-1}+(\xi-\eta)^{-1}]S(x)+[\xi(\xi-\eta)^{-1}-(\xi-\eta)^{-1}]S(y)$. Thus both $u$ and $v$ belong to $l(S(x),S(y))$, which implies that $l(u,v)\subset l(S(x),S(y))$.

 This completes the proof.
\end{proof}

We can now state and prove Propositions \ref{line-line2} and \ref{line-line}.

\begin{proposition}\label{line-line2}
 Let $V$ and $V^\prime$ be two regular $L^0({\mathcal F})$-modules such that $V$ is free with $rank(V)=2$. If $T: V\to V^\prime$ is stable, injective and maps each $L^0$-line to an $L^0$-line, then $T$ must be $L^0$-affine.
\end{proposition}

\begin{proof}
Define $S: V\to V^\prime$ by $S(x)=T(x)-T(0), \forall x\in V$, then $S(0)=0$. With the assumptions on $T$, $S$ is stable and injective. By Proposition \ref{line-line-exact}, $S$ maps each $L^0$-line $l(x,y)$ to the $L^0$-line $l(S(x),S(y))$. It remains to show that $S$ is $L^0$-linear. The proof is divided into 4 steps. We assume that $\{e_1,e_2\}$ is a basis of $V$ and point out in advance that (3) of Proposition \ref{aff-lc} is frequently used.

{\em Step 1.} For any $L^0({\mathcal F})$-independent elements $x$ and $y$ in $V$, we have that $S(x)$ and $S(y)$ are $L^0({\mathcal F})$-independent and $S(x+y)=S(x)+S(y)$.

For any $z\in V$ with full support, let $A=[S(z)=0]$, then by (3) of Proposition \ref{aff-lc}, $S(I_A z)=I_AS(z)=0$. Since $S$ is injective, we obtain that $I_A z=0$. Thus $I_A=0$, implying that $S(z)$ has full support.

Suppose that $\xi$ and $\eta$ are two elements in $L^0({\mathcal F})$ such that $\xi S(x)+\eta S(y)=0$. Since $S$ is injective and maps the $L^0$-line $l(0,x)$ to the $L^0$-line $l(0, S(x))$, there exists $\alpha\in L^0({\mathcal F})$ such that $\xi S(x)=S(\alpha x)$. Similarly, there exists $\beta\in L^0({\mathcal F})$ such that $-\eta S(y)=S(\beta y)$. By the injectivity of $S$ we get $\alpha x=\beta y$, then $\alpha=\beta=0$ since $x$ and $y$ are $L^0({\mathcal F})$-independent. As a result, $\xi S(x)=-\eta S(y)=0$, then using the fact that both $S(x)$ and $S(y)$ have full support, we conclude that $\xi=\eta=0$, which means that $S(x)$ and $S(y)$ are $L^0({\mathcal F})$-independent.

We then show that there exist $a$ and $b$ in $L^0({\mathcal F})$ such that $S(x+y)=aS(x)+bS(y)$. In fact, since $x+y$ lies in the $L^0$-line $l(2x,2y)$, there exists $\mu\in L^0({\mathcal F})$ such that $S(x+y)=\mu S(2x)+(1-\mu)S(2y)$. Since $2x$ lies in the $L^0$-line $l(0,x)$ and $2y$ in the $L^0$-line $l(0,y)$, there exist two elements $\alpha_1$ and $\beta_1$ in $L^0({\mathcal F})$ such that $S(2 x)=\alpha_1S(x)$ and $S(2y)=\beta_1S(y)$, then $a=\mu\alpha_1$ and $b=(1-\mu)\beta_1$ satisfy $S(x+y)=aS(x)+bS(y)$.

We claim that $a=1$ and $b=1$. We prove it by contradiction. If $a\neq 1$, let $A=[a\neq 1]$ and $c_1=-(a-1)^{-1}$,
 then $I_A\neq 0$ and $I_A[1+c_1(a-1)]=0$. Since the $L^0$-line $l(x,x+y)=\{x+cy: c\in L^0({\mathcal F})\}$ is mapped by $S$ to the $L^0$-line $l(S(x),S(x+y))$, there exists $c_0\in L^0({\mathcal F})$ such that $S(x+c_0y)=(1-c_1)S(x)+c_1S(x+y)=[1+c_1(a-1)]S(x)+c_1bS(y)$. Using (3) of Proposition \ref{aff-lc} we obtain that $S(I_A(x+c_0y))=I_AS(x+c_0y)=I_Ac_1b S(y)$. Note that there exists some $\xi\in L^0({\mathcal F})$ such that $I_Ac_1b S(y)=S(\xi y)$, then by the injectivity of $S$, we get $I_A(x+c_0y)=\xi y$, contradicting to the assumption that $x$ and $y$ are $L^0({\mathcal F})$-independent. Therefore, $a=1$. Similarly, $b=1$.

{\em Step 2.} For any two $L^0({\mathcal F})$-independent elements $x$ and $y$ in $V$, we have $S(\xi x+\eta y)=S(\xi x)+S(\eta y), \forall\xi,\eta\in L^0({\mathcal F})$. Specially, $S(\xi e_1+\eta e_2)=S(\xi e_1)+S(\eta e_2), \forall \xi,\eta\in L^0({\mathcal F})$.

First suppose that $\xi=\tilde I_A$ and $\eta=\tilde I_B$ for some $A$ and $B$ in ${\mathcal F}$. Since
$\tilde I_A x+\tilde I_By=\tilde I_{A\cap B}(x+y)+\tilde I_{A\setminus B}x+\tilde I_{B\setminus A}y$, we have that $S(\tilde I_A x+\tilde I_By)=\tilde I_{A\cap B}S(x+y)+\tilde I_{A\setminus B}S(x)+\tilde I_{B\setminus A}S(y)=\tilde I_{A\cap B}[S(x)+S(y)]+\tilde I_{A\setminus B}S(x)+\tilde I_{B\setminus A}S(y)=\tilde I_A S(x)+\tilde I_BS(y)=S(\tilde I_A x)+S(\tilde I_By)$.

Generally, for any $\xi$ and $\eta$ in $L^0({\mathcal F})$, let $A=[\xi\neq 0]$ and $B=[\eta\neq 0]$, and take $x_1=\xi x+I_{A^c}x$ and $y_1=\eta y+I_{B^c}y$. Then $\xi x=I_Ax_1$, $\eta y=I_B y_1$, and $x_1$ and $y_1$ are $L^0({\mathcal F})$-independent. In fact, if $\alpha$ and $\beta$ are elements in $L^0({\mathcal F})$ such that $\alpha x_1+\beta y_1=0$, now that $x$ and $y$ are $L^0({\mathcal F})$-independent, we deduce that $\alpha(\xi+I_{A^c})=0$ and $\beta(\eta+I_{B^c})=0$. Note that $\xi+I_{A^c}\neq 0$ on $\Omega$ and $\eta+I_{B^c}\neq 0$ on $\Omega$, we thus obtain $\alpha=\beta=0$. Now we have shown that $x_1$ and $y_1$ are $L^0({\mathcal F})$-independent, then $S(\xi x+\eta y)=S(I_A x_1+I_By_1)=S(\tilde I_A x_1)+S(\tilde I_By_1)=S(\xi x)+S(\eta y)$.

{\em Step 3.} For each $i\in \{1,2\}$, we have $S(\xi e_i+\eta e_i)=S(\xi e_i)+S(\eta e_i), \forall \xi,\eta\in L^0({\mathcal F})$.

By symmetry, it suffices to prove the case when $i=1$.

Since $e_1-e_2$ and $e_2$ are obviously $L^0({\mathcal F})$-independent, we get from Step 2 that $S(e_1)=S(e_1-e_2+ e_2)=S(e_1-e_2)+S(e_2)=S(e_1)+S(-e_2)+S(e_2)$, therefore $S(e_2)+S(-e_2)=0$.

Now fix two elements $\xi$ and $\eta$ in $L^0({\mathcal F})$, let $A=[\xi+\eta\neq 0]$. Then $x_1=\xi e_1+I_{A^c}e_1+e_2$ and $y_1=\eta e_1-e_2$ are $L^0({\mathcal F})$-independent. Indeed, if $\alpha$ and $\beta$ are elements in $L^0({\mathcal F})$ such that $\alpha x_1+\beta y_1=(\alpha \xi+\alpha I_{A^c}+\beta\eta)e_1+(\alpha-\beta)e_2=0$, then $\alpha \xi+\alpha I_{A^c}+\beta\eta=0$ and $\alpha-\beta=0$, equivalently $\alpha=\beta$ and $\alpha(\xi+\eta+I_{A^c})=0$. Noting that $\xi+\eta+I_{A^c}\neq 0$ on $\Omega$, we thus obtain $\alpha=\beta=0$. From Step 2, noting that $e_1$ and $I_{A^c}e_1+e_2$ are $L^0({\mathcal F})$-independent, we get $S(x_1)=S(\xi e_1)+S(I_{A^c}e_1+e_2)=S(\xi e_1)+S(I_{A^c}e_1)+S(e_2)$ and $S(y_1)=S(\eta e_1)+S(-e_2)$. On the other hand, since $x_1$ and $y_1$ are $L^0({\mathcal F})$-independent, it follows from Step 2 that $S(\xi e_1+I_{A^c} e_1+\eta e_1)=S(x_1+y_1)=S(x_1)+S(y_1)$. Therefore, using $S(e_2)+S(-e_2)=0$ we get $S(\xi e_1+I_{A^c} e_1+\eta e_1)=S(\xi e_1)+S(I_{A^c}e_1)+S(\eta e_1)$. Using (3) of Proposition \ref{aff-lc}, we obtain that $I_{A^c}S(\xi e_1+I_{A^c} e_1+\eta e_1)=S[I_{A^c}(\xi e_1+I_{A^c} e_1+\eta e_1)]=S(I_{A^c} e_1)$,
hence $S(\xi e_1+\eta e_1)=S[I_A(\xi e_1+I_{A^c} e_1+\eta e_1)]=I_AS(\xi e_1+I_{A^c} e_1+\eta e_1)=S(\xi e_1+I_{A^c} e_1+\eta e_1)-I_{A^c}S(\xi e_1+I_{A^c} e_1+\eta e_1)=S(\xi e_1)+S(\eta e_1)$.

{\em Step 4}. For each $i\in \{1,2\}$, we have $S(\xi e_i)=\xi S(e_i), \forall \xi\in L^0({\mathcal F})$.

 Fix an $x$ in $V$ with full support. For any $\xi\in L^0({\mathcal F})$, since $\xi x$ lies in the $L^0$-line $l(0,x)$, there exists $\mu \in L^0({\mathcal F})$ such that $S(\xi x)=\mu S(x)$. By Step 1, $S(x)$ has full support, thus $\mu$ is uniquely determined by $\xi$ (and $x$).
Therefore we can define a mapping $f_x:L^0({\mathcal F})\to L^0({\mathcal F})$ by the relation $S(\xi x)=f_x(\xi)S(x), \forall \xi \in L^0({\mathcal F})$.

Specially, for each $i\in \{1,2\}$, we have a mapping $f_i: L^0({\mathcal F})\to L^0({\mathcal F})$ such that $S(\xi e_i)=f_i(\xi) S(e_i), \forall \xi \in L^0({\mathcal F})$.

 We show that $f_1=f_2$. For each $\xi\in L^0({\mathcal F})$, since $e_1+e_2$ has full support, we have $S(\xi(e_1+e_2))=f_{e_1+e_2}(\xi)S(e_1+e_2)=f_{e_1+e_2}(\xi)[S(e_1)+S(e_2)]$, where the last equality follows from Step 1. From Step 2, $S(\xi(e_1+e_2))=S(\xi e_1)+S(\xi e_2)=f_1(\xi)S(e_1)+f_2(\xi)S(e_2)$. Thus we obtain $f_{e_1+e_2}(\xi)[S(e_1)+S(e_2)]=f_1(\xi)S(e_1)+f_2(\xi)S(e_2)$. We have known from Step 1 that $S(e_1)$ and $S(e_2)$ are $L^0({\mathcal F})$-independent, thus $f_1(\xi)=f_{e_1+e_2}(\xi)=f_2(\xi)$.

Please note that using a similar argument, for any $\eta\in L^0({\mathcal F})$ such that $\eta\neq 0$ on $\Omega$, we have $f_{\eta e_1}(\xi)=f_2(\xi)=f_1(\xi), \forall \xi\in L^0({\mathcal F})$.

We proceed to show that $f_1(\xi)=\xi, \forall \xi\in L^0({\mathcal F})$.

First, it is obvious that $f_1(0)=0$ and $f_1(1)=1$. Then by (3) of Proposition \ref{aff-lc}, $S(\tilde I_A\xi e_1)=\tilde I_AS(\xi e_1)$ for any $\xi\in L^0({\mathcal F})$ and $A\in {\mathcal F}$, implying that $f_1(\tilde I_A\xi)=\tilde I_Af_1(\xi)$, which means that $f_1$ is local. By Step (3), $S(\xi e_1+\eta e_1)=S(\xi e_1)+S(\eta e_1)$ for any $\xi$ and $\eta$ in $L^0({\mathcal F})$, implying that $f_1(\xi+\eta)=f_1(\xi)+f_1(\eta), \forall \xi,\eta\in L^0({\mathcal F})$. Finally, for any $\xi$ and $\eta$ in $L^0({\mathcal F})$, choose $\eta_1\in L^0({\mathcal F})$ such that $\eta_1\neq 0$ on $\Omega$ and $\eta=I_A\eta_1$, where $A=[\eta\neq 0]$ (for instance, we can take $\eta_1=I_A\eta+I_{A^c}$), then we have $S((\xi\eta)e_1)=f_1(\xi\eta)S(e_1)$, and
\begin{align*}
  S((\xi\eta)e_1) & =S(\xi I_A\eta_1e_1) \\
   & =I_AS(\xi \eta_1e_1) \\
   & =I_Af_{\eta_1e_1}(\xi)S(\eta_1 e_1) \\
   & =I_Af_1(\xi)f_1(\eta_1)S(e_1) \\
   & =f_1(\xi)f_1(\eta)S(e_1).
\end{align*}
Noting that $S(e_1)$ has full support, we thus obtain $f_1(\xi\eta)=f_1(\xi)f_1(\eta)$.

To sum up, $f_1$ satisfies all the conditions (1-4) in Lemma \ref{id}, thus $f_1(\xi)=\xi, \forall \xi\in L^0({\mathcal F})$.

Combining Step 2 and Step 4, for any $\xi$ and $\eta$ in $L^0({\mathcal F})$, we have $S(\xi e_1+\eta e_2)=S(\xi e_1)+S(\eta e_2)=\xi S(e_1)+\eta S(e_2)$, meaning that $S$ is $L^0$-linear.

This completes the proof.
\end{proof}

\begin{proposition}\label{line-line}
 Let $V$ and $V^\prime$ be two regular $L^0({\mathcal F})$-modules such that $V$ contains a free $L^0({\mathcal F})$-submodule of rank $2$. If $T: V\to V^\prime$ is stable, injective and maps each $L^0$-line to an $L^0$-line, then $T$ must be $L^0$-affine.
\end{proposition}

\begin{proof}
 Suppose that $x_0$ and $y_0$ are $L^0({\mathcal F})$-independent elements in $V$. Fix $x$ and $y$ in $V$, we first prove that there exists an $L^0({\mathcal F})$-submodule $V_1$ of $V$ such that $V_1$ is free with $rank(V_1)= 2$ and contains $x$ and $y$.

 Consider the $L^0({\mathcal F})$-submodule $U$ of $V$ generated by $x$ and $y$, namely $U=\{\xi x+\eta y: \xi, \eta\in L^0({\mathcal F})\}$, according to Proposition \ref{atfgm}, there exists a partition $\{A_0,A_1,A_2\}$ of $\Omega$ to $\mathcal F$ such that $\tilde I_{A_i}U$ is a free module of rank $i$ over the algebra $\tilde I_{A_i}L^0({\mathcal F})$ for each $i\in \{0,1,2\}$ satisfying $P(A_i)>0$, in which case $U=\bigoplus^2_{i=0}\tilde I_{A_i}U$. We may without loss of generality assume that $P(A_0)>0$, $P(A_1)>0$ and $P(A_2)>0$. Suppose that $\tilde I_{A_1}z$ is a basis of the free $\tilde I_{A_1}L^0({\mathcal F})$-module $\tilde I_{A_1}U$ of rank 1. Consider the $\tilde I_{A_1}L^0({\mathcal F})$-module $W$ generated by $\tilde I_{A_1}z$ and $\tilde I_{A_1}x_0$, again by Proposition \ref{atfgm} there exists a partition $\{B_0,B_1,B_2\}$ of $A_1$ to $\mathcal F$ such that $\tilde I_{B_i}W$ is a free module of rank $i$ over the algebra $\tilde I_{B_i}L^0({\mathcal F})$ for each $i\in \{0,1,2\}$ satisfying $P(B_i)>0$. Note that $\tilde I_{B_0}W=\{0\}$, specially $\tilde I_{B_0}x_0=0$, it follows that $\tilde I_{B_0}=0$, namely $P(B_0)=0$. We suppose both $P(B_1)>0$ and $P(B_2)>0$. Note in such a case, there exists a $\xi_0\in L^0({\mathcal F})$ with $\xi_0\neq 0$ on $B_1$ such that $I_{B_1}z=I_{B_1}\xi_0 x_0$. Take $x_1=\tilde I_{A_0}x_0+\tilde I_{A_1}z+\tilde I_{A_2}x$ and $y_1=\tilde I_{A_0}y_0+
\tilde I_{B_1}y_0+\tilde I_{B_2}x_0+\tilde I_{A_2}y$, we claim that $x_1$ and $y_1$ are $L^0({\mathcal F})$-independent, and $U$ a subset of $V_1:=\{\xi x_1+\eta y_1:\xi, \eta\in L^0({\mathcal F})\}$. In fact, on $A_0$, we have $\tilde I_{A_0}x_1=\tilde I_{A_0}x_0$ and $\tilde I_{A_0}y_1=\tilde I_{A_0}y_0$, thus $\tilde I_{A_0}x_1$ and $\tilde I_{A_0}y_1$ are $\tilde I_{A_0}L^0({\mathcal F})$-independent, and $\tilde I_{A_0}U=\tilde I_{A_0} \{0\}\subset \tilde I_{A_0}V_1$; on $A_1$, we have $\tilde I_{A_1}x_1=\tilde I_{A_1}z=\tilde I_{B_1}z+\tilde I_{B_2}z=I_{B_1}\xi_0 x_0+\tilde I_{B_2}z$ and $\tilde I_{A_1}y_1=\tilde I_{B_1}y_0+\tilde I_{B_2}x_0$, thus $\tilde I_{A_1}x_1$ and $\tilde I_{A_1}y_1$ are $\tilde I_{A_1}L^0({\mathcal F})$-independent, and $\tilde I_{A_1}U=span_{\tilde I_{A_1}L^0({\mathcal F})}\{\tilde I_{A_1}z\}\subset I_{A_1}V_1$; on $A_2$, we have $\tilde I_{A_2}x_1=\tilde I_{A_2}x$ and $\tilde I_{A_2}y_1=\tilde I_{A_2}y$, thus $\tilde I_{A_2}x_1$ and $\tilde I_{A_2}y_1$ are $\tilde I_{A_2}L^0({\mathcal F})$-independent, and $\tilde I_{A_2}U=\tilde I_{A_2}V_1$.

Now $V_1$ is a regular and free $L^0({\mathcal F})$-module of rank 2, and $x$ and $y$ are elements in $V_1$. Consider the restriction of $T$ to $V_1$,  then $T$ is $L^0$-affine on $V_1$ by Proposition \ref{line-line2}. It follows from (1) of Proposition \ref{aff-lc} that $T(\lambda x+(1-\lambda)y)=\lambda Tx+(1-\lambda)Ty, \forall \lambda\in L^0({\mathcal F})$.

This completes the proof.
\end{proof}

In the following, we give an example which shows that a bijective mapping $T: L^0({\mathcal F},\mathbb R^n)\to L^0({\mathcal F},\mathbb R^n)$, which maps any $L^0$-line to an $L^0$-line, may be not stable, thus according to (2) of Proposition \ref{aff-lc}, this mapping $T$ is not $L^0$-affine.

\begin{example}\label{exam}

Let $\theta: (\Omega,{\mathcal F},P)\to (\Omega,{\mathcal F},P)$ be an isomorphism, that is to say, $\theta$ is bijective and both $\theta$ and $\theta^{-1}$ are measure-preserving. Then $\theta$ induces a bijection $\sigma: L^0({\mathcal F})\to L^0({\mathcal F})$ sending $\xi$ to the equivalence class of $\xi^0(\theta(\cdot))$, where $\xi^0$ is a representative of $\xi\in L^0({\mathcal F})$. Further, for each positive integer $n$, $\theta$ induces a bijection $T: L^0({\mathcal F},\mathbb R^n)\to L^0({\mathcal F},\mathbb R^n)$ sending $(\xi_1,\dots,\xi_n)$ to $ (\sigma(\xi_1),\dots,\sigma(\xi_n))$. Then it is straightforward to check that $T$ maps each $L^0$-line to an $L^0$-line. However, if $\theta$ is not the identity mapping, $T$ is probably not stable.

Following is a more concrete example.

Let $\Omega$ be $[0,1)$, ${\mathcal F}$ the Borel $\sigma$-algebra of $[0,1)$, and $P$ the Lebesgue measure.
Define $\theta: [0,1)\to [0,1)$ by $\theta(\omega)=\omega+\frac{1}{2}$ for $\omega\in [0,\frac{1}{2})$, and $\theta(\omega)=\omega-\frac{1}{2}$ for $\omega\in [\frac{1}{2}, 1)$. We show that as stated above, the induced mapping $T: L^0({\mathcal F},\mathbb R^n)\to L^0({\mathcal F},\mathbb R^n)$ is not stable. Let $A=[0,\frac{1}{2})$, $B=[\frac{1}{2}, 1)$, then $\sigma(\tilde I_A)=\tilde I_B$, and thus for each $x\in L^0({\mathcal F},\mathbb R^n)$ we have $T(\tilde I_A x)=\tilde I_B T(x)$, specially $T(\tilde I_A e_1)=\tilde I_B T(e_1)=\tilde I_B e_1$, where $e_1=(1,0,\dots,0)$, it follows that $\tilde I_AT(\tilde I_A e_1)=0\neq \tilde I_Ae_1=\tilde I_AT(e_1)$. Thus $T$ is not stable.
\end{example}

\begin{remark}
 Since $L^0({\mathcal F})$ is a commutative algebra and $L^0({\mathcal F})\neq \{0\}$, according to Theorem 2.6 of \cite{Cohn}, $L^0({\mathcal F})$ is an IB-ring (see \cite{LK} for details). Lashkhi and Kvirikashvili \cite{LK} has established fundamental theorem of affine geometry of modules over IB-rings. It is necessary to give a comparison. Applying Theorem 1 of \cite{LK} to regular $L^0({\mathcal F})$-modules gives the following fact: let $V$ and $V^\prime$ be two regular free $L^0({\mathcal F})$-modules such that $rank(V)\geq 2$, if $T: V\to V^\prime$ with $T(0)=0$ is a collineation preserving parallelism, that is to say, $T$ is a bijection such that the images of collinear points under $T$ are themselves collinear and $T$ preserves parallelism (see \cite{LK} for this notion), then there exists an isomorphism $\sigma: L^0({\mathcal F})\to L^0({\mathcal F})$ such that $T$ is a
$\sigma$-semilinear isomorphism, namely $T(x+y)=T(x)+T(y)$ for any $x$ and $y$ in $V$ and $T(\xi x)=\sigma(\xi)T(x)$ for any $x\in V$ and $\xi\in L^0({\mathcal F})$. In this case, Theorem 1 of \cite{LK} requires that $V$ and $V^\prime$ be free and $T$ be bijective and parallelism preserving, whereas our Proposition \ref{line-line} only requires a simple condition--$T$ being stable, therefore our Proposition \ref{line-line} is not a special case of Theorem 1 of \cite{LK}. We also would like to point out that although $T$ in Example \ref{exam} is a $\sigma$-semilinear isomorphism, it is not an $L^0$-linear mapping.
\end{remark}

With Proposition \ref{line-line}, to prove Theorem \ref{main} we only need to show that a stable and bijective mapping, which maps each $L^0$-line segment to an $L^0$-line segment, must maps each $L^0$-line to an $L^0$-line.

\begin{proposition}\label{line-seg}
  Let $V$ and $V^\prime$ be two regular $L^0({\mathcal F})$-modules such that $V$ contains an element $e$ with full support. If $T: V\to V^\prime$ is bijective, stable and maps each $L^0$-line segment to an $L^0$-line segment, then $T$ maps each $L^0$-line to an $L^0$-line.
\end{proposition}

\begin{proof}
 We can without loss of generality assume that $T(0)=0$, otherwise we make a translation. Let $x$ and $y$ be any two elements in $V$ such that $y\neq 0$. Since $T$ is a bijection, it follows from (5) of Proposition \ref{aff-lc} that $T^{-1}$ is stable, and by Proposition \ref{line-line-exact} we can see that $T^{-1}$ also maps each $L^0$-line segment to an $L^0$-line segment, thus we only need to show that each point $z$ in the $L^0$-line $l(x,x+y)=\{x+\lambda y: \lambda\in L^0({\mathcal F})\}$ will be mapped into the $L^0$-line $l(T(x),T(x+y))=\{\lambda T(x)+(1-\lambda)T(x+y): \lambda\in L^0({\mathcal F})\}$.

We first show that for each $k\in \mathbb{Z}=\{0,\pm 1,\pm 2,\dots\}$, $z=x+ky$ will be mapped into $l(T(x),T(x+y))$.

First assume that $y$ has full support. (1) The cases $k=0$ and $k=1$ are obvious. (2) Fix a $k\in\{2,3,4,\dots\}$. Since $x+y=(1-\frac{1}{k})x+\frac{1}{k}(x+ky)\in [x,x+ky]$ and $T$ maps an $L^0$-line segment to an $L^0$-line segment, there exists $\mu\in L^0({\mathcal F})$ with $0\leq \mu\leq 1$ such that $T(x+y)=(1-\mu) T(x)+\mu T(x+ky)$. Let $A=[\mu=0]$, then $I_A\mu=0$. By (3) of Proposition \ref{aff-lc}, $T(I_A(x+y))=I_AT(x+y)=I_A[(1-\mu) T(x)+\mu T(x+ky)]=I_AT(x)=T(I_Ax)$. Since $T$ is a bijection, we obtain $I_A(x+y)=I_A x$. Then $I_A=0$ follows from the assumption that $y$ has full support, equivalently $\mu>0$ on $\Omega$. As a result, $T(x+ky)=\frac{1}{\mu}T(x+y)+(1-\frac{1}{\mu})T(x)\in l(T(x), T(x+y))$. (3) Fix a $k\in\{-1,-2,-3,\dots,\}$. Since $x=\frac{1}{1-k}(x+ky)+(1-\frac{1}{1-k})(x+y)\in [x+ky, x+y]$, there exists $\mu\in L^0({\mathcal F})$ with $0\leq \mu\leq 1$ such that $T(x)=\mu T(x+ky)+(1-\mu)T(x+y)$, by a similar argument as in (2) we deduce that $\mu>0$ on $\Omega$, then $T(x+ky)=\frac{1}{\mu}T(x)+(1-\frac{1}{\mu})T(x+y)\in l(T(x), T(x+y))$.

 Now for a general nonzero $y$. Take $y_1=I_Ay+I_{A^c}e$, where $A=[y\neq 0]$, we see that $y_1$ has full support and $y=I_Ay_1$. Fix any $k\in {\mathbb Z}$, we have proved that there exists $\mu \in L^0({\mathcal F})$ such that $T(x+ky_1)=\mu T(x)+(1-\mu)T(x+y_1)$. Using (4) of Proposition \ref{aff-lc}, $T(x+ky)=T[I_A(x+ky)]+T[I_{A^c}(x+ky)]=T[I_A(x+ky_1)]+T(I_{A^c}x)=I_A[\mu T(x)+(1-\mu)T(x+y_1)]+I_{A^c}T(x)=(I_A\mu+I_{A^c})T(x)+I_A(1-\mu)T(x+y)$, implying that $T(x+ky)\in l(T(x),T(x+y))$.

We then show that for each $\lambda\in L^0({\mathcal F})$, $z=x+\lambda y$ will be mapped into $l(T(x),T(x+y))$.

For $k=1,2,\dots$, let $A_k=[k-1\leq |\lambda| <k]$, then $x+\lambda I_{A_k}y=(\frac{1}{2}-\frac{\lambda}{2k}I_{A_k})(x-ky)+(\frac{1}{2}+\frac{\lambda}{2k}I_{A_k})(x+ky)$ belongs to $[x-ky, x+ky]$, consequently $T(x+\lambda I_{A_k}y)\in [T(x-ky), T(x+ky)]\subset l(T(x),T(x+y))$, where the last inclusion follows from the fact that both the two endpoints $T(x-ky)$ and $T(x+ky)$ belong to $l(T(x),T(x+y))$. Thus for each positive integer $k$, there exists $\mu_k\in L^0({\mathcal F})$ such that $T(x+\lambda I_{A_k}y)=\mu_kT(x)+(1-\mu_k)T(x+y)$. By (3) of Proposition \ref{aff-lc}, we have that $I_{A_k}T(x+\lambda y)=I_{A_k}T(x+\lambda I_{A_k}y)=I_{A_k}[\mu_kT(x)+(1-\mu_k)T(x+y)]$ for each $k$.
Let $\mu=\sum^{\infty}_{k=1}I_{A_k}\mu_k$, then $I_{A_k}T(x+\lambda y)=I_{A_k}[\mu T(x)+(1-\mu)T(x+y)]$ for each $k$. By the regularity of $V^\prime$, we conclude that $T(x+\lambda y)=\mu T(x)+(1-\mu)T(x+y)$, which means that $T(x+\lambda y)\in l(T(x),T(x+y))$.

This completes the proof.
\end{proof}

We can now prove Theorem \ref{main}. \\
\begin{proof}
It immediately follows from Propositions \ref{line-line} and \ref{line-seg}.

This completes the proof.
\end{proof}
\section*{Acknowledgements}
The first author was supported by the Natural Science Foundation of China (Grant No.11701531) and the Fundamental Research Funds for the Central Universities, China University of Geosciences (Wuhan) (Grant No. CUGL170820). The second author was supported by the Natural Science Foundation of China(Grant No.11971483). The third author was supported by the Natural Science Foundation of China(Grant No.11501580). The authors would like to thank the reviewers for their valuable suggestions which considerably improve the readability of this paper.

%\section*{References}

\end{document}